\pdfoutput=1
\documentclass[letterpaper, 10 pt, conference]{ieeeconf}  

\IEEEoverridecommandlockouts                              

\usepackage{mathrsfs}
\usepackage[font={small}]{caption}
\usepackage{scalerel}
\usepackage{url}
\usepackage{bbold}
\usepackage{amsfonts}
\usepackage{amsmath,amssymb,amsfonts}
\usepackage{graphicx}
\usepackage{wrapfig}
\usepackage{mathrsfs} 
\usepackage{algorithm,algorithmic}
\usepackage{array}
\usepackage{times}
\usepackage{url}
\usepackage{cite}
\usepackage[thinc]{esdiff}
\newcommand{\citet}[1]{\cite{#1}}
\usepackage{upgreek}
\usepackage{float}
\usepackage{longtable}
\usepackage{color}
\usepackage{wasysym}
\usepackage{grffile}
\usepackage{stackengine}
\usepackage[capitalize,noabbrev]{cleveref}

\usepackage[symbol]{footmisc}

\allowdisplaybreaks

\newcommand{\eb}{\mathbf{e}}

\newcommand{\ub}{\mathbf{u}}
\newcommand{\wb}{\mathbf{w}}
\newcommand{\xb}{\mathbf{x}}

\newcommand{\Ab}{\mathbf{A}}
\newcommand{\Bb}{\mathbf{B}}

\newcommand{\Gb}{\mathbf{G}}
\newcommand{\Hb}{\mathbf{H}}

\newcommand{\Kb}{\mathbf{K}}
\newcommand{\Lb}{\mathbf{L}}
\newcommand{\Mb}{\mathbf{M}}

\newcommand{\Pb}{\mathbf{P}}

\newcommand{\Rb}{\mathbf{R}}

\newcommand{\Wb}{\mathbf{W}}

\newcommand{\Sc}{\mathcal{S}}

\newcommand{\Xc}{\mathcal{X}}

\newcommand{\norm}[1]{\left\lVert#1\right\rVert}

\newtheorem{theorem}{Theorem}

\newtheorem{assumption}{Assumption}

\newtheorem{definition}{Definition}

\newtheorem{lemma}[theorem]{Lemma}

\title{\bf A Spectral Filtering Approach to Regret Analysis of Distributed Online Control for Linear Dynamical Systems}

\author{Ting-Jui Chang
\thanks{T.J. Chang is with the Department of Aeronautics and Astronautics, National Cheng Kung University, Taiwan. 
{\tt\footnotesize email:tjc@gs.ncku.edu.tw}.}%
}

\begin{document}

\maketitle
\thispagestyle{plain}
\pagestyle{plain}

{\color{black}
\begin{abstract}

    This paper studies the distributed online control problem over a network of linear time-invariant (LTI) systems in the presence of adversarial disturbances and time-varying convex costs. The network cost is characterized by the summation of local cost functions, where each local function is sequentially revealed only to the corresponding agent. The goal of each agent is to generate a control sequence, using only local observations and neighbor communication, that competes with the best {\it centralized} linear policy in hindsight. We extend the recently proposed Online Spectral Control framework from the centralized setting to the distributed setting. In particular, each agent applies a spectral controller obtained by convolving past disturbances with the leading eigenvectors of a Hankel matrix, while the controller parameters are updated through a distributed online gradient descent step over the local surrogate costs. We formulate this problem this problem as a {\it regret} minimization problem based on the spectral parameterization, and under standard assumptions, we establish a sublinear regret bound of $O(\frac{\sqrt{T}\text{poly}(\log T)}{\gamma^3})$, where $T$ is the time horizon and $\gamma$ denotes the stability margin. The resulting bound also captures the dependence on the network size and connectivity.

\end{abstract}

}

\section{Introduction}

{\color{black}
In recent years, there has been a growing interest in studying problems at the intersection of control theory, machine learning, and online optimization. Classical optimal control provides a rich collection of tools for designing controllers when the system dynamics and the performance criterion are known in advance. For example, in the linear quadratic regulator (LQR) problem, when the dynamics are given and the cost function is quadratic and known in advance, the optimal controller can be computed through Riccati equations. However, many modern control applications operate in non-stationary environments where the cost metric may vary over time, the disturbance sequence may not follow any statistical model, and future information is unavailable at the time when the controller is applied.

These challenges motivate the study of {\it online} control, where the controller must make sequential decisions while adapting to the revealed cost information. The performance of an online controller is typically measured by {\it regret}, defined as the difference between the cumulative cost incurred by the online controller and that of the best policy in hindsight. A sublinear regret bound implies that the time-averaged performance of the online controller approaches that of the best benchmark policy asymptotically. In the centralized online control setting, recent works \cite{cohen2018online, agarwal2019online, agarwal2019logarithmic} have shown that online learning techniques can be used to design controllers for linear dynamical systems with time-varying convex costs which are unknown in advance. In particular, the Online Spectral Control method \cite{brahmbhatt2025new} proposes a convex relaxation of linear feedback policies by constructing spectral filters from the eigenvectors of a Hankel matrix. This method achieves the same order of regret guarantees as previous disturbance-feedback approaches \cite{agarwal2019online}, while improving the runtime dependence on the inverse stability margin from polynomial to polylogarithmic scaling.

On the other hand, the computational and informational constraints of large-scale control systems motivate the development of distributed control algorithms. In many practical systems, the overall control task is carried out by a network of agents or subsystems, such as unmanned aerial vehicles \cite{buzogany1993automated,wolfe1996decentralized}, mobile robots \cite{yamaguchi1998cooperative,yamaguchi2001distributed,yang2022collaborative} and microsatellite clusters \cite{burns2000techsat,schaub2000spacecraft}. In these applications, it is often unrealistic to assume that a centralized controller has direct access to the complete cost information of all agents at every time step. Instead, each agent observes only its local cost and communicates with its neighbors through a given network topology. The objective is therefore to design local controllers such that each agent, despite having only partial information, can achieve performance comparable to that of a centralized policy that has access to the global network cost.

Motivated by these considerations, this paper studies the distributed online control problem for a network of linear time-invariant systems under adversarial disturbances and time-varying convex costs. Each agent is modeled as an LTI system, and the global network cost is given by the summation of local cost functions. At each round, every agent observes its own state, applies a local control input, receives its local cost function, and exchanges information only with its neighbors. The goal is to minimize the individual regret of each agent with respect to the best centralized policy in hindsight, where the benchmark is chosen from the class of diagonalizably stable linear controllers.

The main difficulty is that the centralized spectral controller cannot be directly applied in the distributed setting. In the centralized case, the controller parameters are updated using the gradient of the global loss. In the distributed case, however, each agent has access only to its local cost function and must rely on communication to approximate the global objective. Moreover, the spectral controller is time-varying because the parameters are updated online, and the actual state depends on the entire history of past controller parameters. This prevents a direct reduction to standard distributed online convex optimization. To address this issue, we introduce a distributed online spectral control framework in which agents update their local spectral parameters through a consensus-based online gradient step over memory-less surrogate costs. The surrogate construction allows the control problem to be reformulated as a distributed online learning problem over a fixed-dimensional convex parameter set, while the spectral filters preserve the approximation power of the centralized spectral controller. With this reparameterization, we are able to analyze the performance of the distributed controller by borrowing the existing techniques from online learning.

Our main contributions are summarized as follows.
\begin{enumerate}
    \item Based on the online spectral controller (OSC) proposed in \cite{brahmbhatt2025new}, we develop the distributed variant called distributed OSC (D-OSC). At each round, every agent applies spectral filters to the past disturbances and computes its action with its local control parameters based on the filtered information. Then, all agents jointly apply a distributed online gradient step (D-OGD) over local surrogate costs (parameterized with the spectral filters) to update the local controllers.

    \item We show that the original distributed online control problem can be transformed into a distributed online learning problem through the parameterization of spectral controllers. By appropriately choosing the hyper-parameters (the learning rate, the window size of past noises and the number of the spectral filters), we derive a sublinear regret bound of $O(\frac{\sqrt{T}\text{poly}(\log T)}{\gamma^3})$ where, as the centralized OSC, the dependence on the stability margin $\gamma$ is improved compared to disturbance feedback controller (DFC) previously proposed in \cite{agarwal2019online}.

    \item With respect to the distributed setup, the dependence of the regret bound on the network connectivity and size is also characterized: $O(\frac{n\sqrt{n}}{(1-\beta)})$, where $n$ denotes the total number of agents and $\beta$ is a parameter inversely proportional to the network connectivity. 
\end{enumerate}

}

\section{Related Literature}

{\color{black}
\textbf{Online Optimization:} Online control is closely related to online convex optimization, where a learner sequentially selects actions before observing the corresponding cost functions. A classical result in this area is online gradient descent, which achieves sublinear regret of $O(\sqrt{T})$ for convex functions \cite{zinkevich2003online} (see e.g., \cite{cesa2006prediction,hazan2016introduction} for more details). In this work, based on the reparameterization using spectral controllers, we transform the original online control problem into a distributed online learning problem which allows us to gauge the control performance through the lens of online learning.  

}

{\color{black}
\textbf{Regret vs. Stability:} The regret of an online controller is not merely an optimization measure; it is also closely related to stability properties of the corresponding closed-loop system. Existing results have shown that, for linear systems driven by adversarial disturbances, regret bounds and stability notions are connected in both time-invariant and time-varying settings regarding the system dynamics \cite{karapetyan2022implications}. This connection provides a control-theoretic justification for using regret as the performance metric. Similar discussion for non-linear systems can also be found in \cite{nonhoff2023relation}.
}

{\color{black}
\textbf{Online Control:}
In contrast to online convex optimization which studies pure optimization problems, online control must additionally account for the effect of current decisions on future states through the system dynamics. This distinction makes online control more challenging. We roughly divide the online control literature into two categories, based on the types of the considered cost functions.

\begin{enumerate}
    \item \textbf{Online LQR Control:} The linear quadratic regulator serves as a canonical optimal control problem, where the stage costs are quadratic functions of the state and control input. When the system dynamics are unknown, the controller must be designed adaptively based on the learned model, and a number of recent works have established sublinear regret guarantees for this setting \cite{dean2018regret,cohen2019learning,cassel2020logarithmic,simchowitz2020naive,lale2022reinforcement}. A related line of research considers systems driven by potentially adversarial disturbances. In this direction, \cite{yu2020power} studied an MPC-type approach for LQR with time-invariant costs under accurate disturbance predictions, while \cite{zhang2021regret} generalized the framework to time-varying quadratic costs with imperfect predictions. Both works analyzed the performance of the proposed methods through dynamic regret bounds. In \cite{cohen2018online}, the setup of time-varying costs and stochastic noises was considered, and a regret bound of $O(\sqrt{T})$ was proved by formulating the problem as a time-varying semi-definite programming (SDP) problem.

    \item \textbf{Online Convex Control:} As opposed to quadratic costs, there is another line of research extending the setup to the case where the time-varying costs are generally convex and the noises are adversarial. Toward this direction, \cite{agarwal2019online} proposed DFC and parameterized the online control problem as an online convex problem with memory for which a regret bound of $O(\sqrt{T})$ was achieved. The bound was further improved to $O(\text{poly}(\text{log}T))$ for strongly convex costs \cite{agarwal2019logarithmic}. This $O(\text{poly}(\text{log}T))$ bound was also proved for the setup where the system is partially observable and the noises are semi-adversarial \cite{simchowitz2020improper}. Recently, \cite{brahmbhatt2025new} proposed the idea of applying spectral filters to approximate linear control policies, which further improved the dependence on the stability margin compared to DFC. Besides the setup where the function information is available, online control with bandit feedback was studied in \cite{sun2023optimal,suggala2024second}.
\end{enumerate}
}

{\color{black}
\textbf{Distributed Control:}
Distributed LQR has received considerable attention in the control literature. Many existing works consider multi-agent systems with known, identical, and dynamically decoupled agents. For instance, \cite{4626964} developed a distributed control synthesis method by solving an LQR problem whose dimension depends on the maximum degree of the underlying communication graph. In \cite{6862471}, the authors characterized necessary conditions for optimal distributed controller design, which leads to a non-convex optimization problem. The work \cite{5299181} studied a multi-agent network in which each agent follows single-integrator dynamics, and showed that computing the optimal controller requires global knowledge of both the network topology and the initial states of all agents. Due to the inherent difficulty of exactly solving the optimal distributed control problem, \cite{8736845} instead derived sufficient conditions for constructing suboptimal distributed controllers. For the unknown dynamics case, \cite{alemzadeh2019distributed} developed a distributed Q-learning algorithm for dynamically decoupled systems. Beyond the setting of identical and decoupled subsystems, several works have also considered more general distributed control problems. For example, \cite{fattahi2019efficient} studied the design of distributed controllers for unknown sparse LTI systems. \cite{furieri2020learning} investigated model-free approaches for distributed LQR and established sample-complexity guarantees for problem classes satisfying a local gradient dominance property, such as quadratically invariant problems. For control problems in non-stationary environments, \cite{chang2021distributed} extended the SDP method proposed in \cite{cohen2018online} to the setup of distributed LQR with known dynamics and achieved the regret bound of the same order as \cite{cohen2018online}. The idea was further generalized to the case with unknown dynamics in \cite{chang2023regret}, and a regret bound of $O(T^{2/3}\log T)$ was derived. 

}
{\color{black}
\textbf{Spectral Filtering:}
Spectral filtering was initially introduced as an improper-learning approach for online prediction in linear dynamical systems (LDS). For LDS prediction with symmetric transition matrices, \cite{hazan2017learning} proposed an efficient algorithm which evades the non-convexity of the original system identification problem by overparameterizing the system through spectral filters. This framework was later extended by \cite{hazan2018spectral} to general latent-state LDSs, where a new convex relaxation was developed to handle non-symmetric dynamics and phase information without explicitly identifying the underlying system. However, this extension incurred dependence on the hidden dimension. More recently, \cite{marsden2025dimension} further advanced this direction by addressing asymmetric and marginally stable LDSs with dimension-free regret guarantees, using spectral filters together with Chebyshev-polynomial-based constructions in the complex plane. In parallel, spectral filtering has also been connected to modern sequence modeling: \cite{agarwal2023spectral} proposed spectral state space models, which use fixed convolutional filters derived from spectral filtering to capture long-range dependencies while retaining robustness to the spectrum and hidden dimension of the underlying dynamics. Besides sequence prediction, the idea of spectral filtering was incorporated into the design of controllers. For offline LQR with unknown dynamics, \cite{arora2018towards} showed that the control problem can be formulated as a convex program without first solving the non-convex system-identification problem. The idea was further extended to the online setup with general convex costs in \cite{brahmbhatt2025new}.
}

\section{Preliminaries and Problem Setup}
\subsection{Notation}
\begin{center}
\begin{tabular}{|c||l|}
    \hline
    $[m]$ & The set $\{1,2,\ldots,m\}$ for any integer $m$ \\
    \hline
    $\norm{\cdot}$ & Euclidean (spectral) norm of a vector (matrix)\\
    \hline
    $\norm{\cdot}_F$ & Frobenius norm of a matrix\\
    \hline
   $\Pi_{\Xc}[\cdot]$ & The operator for the projection to set $\Xc$ \\
    \hline
    $[\Ab]_{ij}$ & The entry in the $i$-th row and $j$-th column of $\Ab$\\
    \hline
    $[\Ab]_{i,:}$ & The $i$-th row of $\Ab$\\
    \hline
    $[\Ab]_{:,j}$ & The $j$-th column of $\Ab$\\
    \hline
    $\mathbf{1}$ & The vector of all ones\\
    \hline
    $\eb_i$ & The $i$-th basis vector\\
    \hline
    $\mathbb{1}$ & Indicator function\\
    \hline
\end{tabular}
\end{center}

\subsection{Dynamical System and Cost Model}
{\color{black}
We study a network composed of m identical agents, each governed by a LTI system. The evolution of the state for agent $i$ follows
$$\xb_{i,t+1} = \Ab\xb_{i,t}+\Bb\ub_{i,t}+\wb_{t},\quad i\in [n],$$
where $\xb_{i,t}\in\mathrm{R}^{d_1}$ and $\ub_{i,t}\in\mathrm{R}^{d_2}$ denote the state and control input at time $t$, respectively. The disturbance sequence $\{\wb_t\}$ is bounded ($\wb_t\leq W,\;\forall t$), unknown ahead of time, and may be chosen adversarially. In this work, we consider the case where the agent has access to the system dynamics, in other words, $(\Ab,\Bb)$ are known.

The distributed online control process unfolds sequentially. At each time step $t$, agent $i$ observes its current state $\xb_{i,t}$, selects an action $\ub_{i,t}$, and subsequently receives a local cost function $f_{i,t}(\cdot,\cdot)$, incurring the cost $f_{i,t}(\xb_{i,t}, \ub_{i,t})$. Although costs are revealed locally, the objective is global: agents aim to minimize the aggregate network cost
$f_t(\cdot,\cdot) = \sum_{i=1}^n f_{i,t}(\cdot,\cdot)$. Thus, the design of the control policy is cooperative, requiring each agent to perform well with respect to the overall system objective rather than just its local cost. 

To evaluate performance, we introduce a centralized benchmark. For any policy $\pi$, its cumulative cost over a horizon of length $T$ is
\begin{equation}\label{Eq: Centralized benchmark}
    J_T(\pi)=\sum_{t=1}^T c_t(\xb^{\pi}_t, \ub^{\pi}_t).
\end{equation}
An online distributed algorithm is assessed by comparing its performance against the best policy in hindsight from a predefined benchmark class $\Pi$. This comparison is formalized through the notion of individual regret, defined for agent $j$ as
\begin{align}\label{Eq: Individual Regret}
    \text{Regret}_T^j(\mathcal{A}):=J_T^j(\mathcal{A})-\min_{\pi\in\Pi}J_T(\pi),
\end{align}
where $J_T^j(\mathcal{A})$ denotes the cumulative network cost evaluated along the trajectory generated by agent $j$ under algorithm $\mathcal{A}$. A desirable algorithm ensures that this regret grows sublinearly in $T$, implying that its average performance converges to that of the optimal centralized policy over time.

Since the objective depends on the global cost, agents must exchange information. This interaction is modeled by an undirected communication graph $\mathcal{G}=(\mathcal{V},\mathcal{E})$, where $\mathcal{V}=[n]$ denotes the set of nodes (i.e., agents) and $\mathcal{E}$ represents the set of edges. The interaction weights are encoded in a matrix $\Pb$ with elements $\geq 0$, which is assumed to be symmetric, doubly stochastic, and connected. By further assuming $\Pb$ has a positive diagonal, the network exhibits a mixing property that governs how quickly information propagates across agents, which plays a key role in the analysis of distributed algorithms:
\begin{equation}
   \sum_{j=1}^n \left|[\Pb^k]_{ji}-1/n\right|\leq \sqrt{n}\beta^k,\:i\in [n],  
\end{equation}
where $\beta$ is the second largest singular value of $\Pb$.
}

{\color{black}
\textbf{Practical Examples:} Our framework can be used to model distributed control problems where the main system is composed of a network of sub-systems with identical and decoupled dynamics \cite{burns2000techsat,buzogany1993automated,yang2022collaborative}, and the general goal is to minimize the network cost using only local information exchange while reaching consensus. Control problems of this kind have been well-studied through the lens of LQR \cite{5299181,4626964,6862471} while in this work, we further extend to the case where the cost is assumed to be general convex.
\begin{enumerate}
    \item The considered framework can be applied for analyzing the online temperature control of a heating, ventilation and air conditioning (HVAC) system \cite{shin2023near,patel2019economic} in a large building with multiple rooms modeled as a network of sub-systems. Supposing the dynamics of each sub-system can be well approximated by an identical dynamical system, the goal, under the condition that all rooms can only exchange the cost information locally, is to minimize the total energy cost over time (the summation of local energy costs of all rooms) while making the temperature of all rooms reach consensus. By applying our proposed algorithm which guarantees an sub-linear bound on \eqref{Eq: Individual Regret} $\forall i$, this goal can be achieved.

    \item Another example is the consensus control problem of mobility cost in mobile sensor networks (MSNs) [4] in the time-varying setup. Considering a MSN where the time-varying network mobility cost of sensors is modeled as $c_t$ at time $t$. Each agent has a local budget of $c_{i,t}$ and the team goal, under the condition that the information can only be exchanged locally, is to design actions that minimize the global network cost over time while reaching consensus. As our regret guarantee applies to the individual regret defined by any agent and the learning process complies with the network topology, the goal of a MSN can be achieved with our proposed method.
\end{enumerate}
}

\subsection{Strong Stability and Strong Controllability}
{\color{black}
In this work, as \cite{brahmbhatt2025new}, the centralized benchmark we consider is a set of $(\kappa,\gamma)$-diagonalizebaly stable policies defined as follows:

\begin{definition}\label{D: Diagonalizable Stability}
    A linear policy $\Kb$ is $(\kappa, \gamma)$-diagonalizably stable ($\kappa>0$ and $0<\gamma\leq 1$) for the LTI system $(\Ab,\Bb)$ if there exists matrices $\Lb$ and $\Hb$ such that $\Ab+\Bb\Kb = \Hb\Lb\Hb^{-1}$ and the following conditions hold:
    \begin{enumerate}
        \item $\Lb$ is diagonal with {\color{black} nonnegative} entries\footnote[1]{The requirement of nonnegative entries can be further relaxed, and is imposed here for the ease of the analysis.}.
        \item $\norm{L}\leq 1-\gamma$\footnote[2]{For the ease of analysis, $\gamma$ is assumed to be less than $2/3$ without loss of generality.}.
        \item $\norm{\Kb},\norm{H},\norm{H^{-1}}\leq \kappa$.
    \end{enumerate}
\end{definition}
The notion of $(\kappa,\gamma)$-diagonalizable stability is similar to the idea of strong stability proposed in \cite{cohen2018online}, which provides a quantification for a stable linear controller $\Kb$ with respect to the system $(\Ab,\Bb)$ in the sense that any stable linear policy is strongly stable with some $\kappa$ and $\gamma$. Despite the requirement of the diagonal structure imposed on $\Lb$, based on Ackermann's formula \cite{ackermann1972entwurf}, there always exists a $\Kb\in \mathcal{S}$ ($\mathcal{S}=\{\Kb: \Kb \text{ is }(\kappa,\gamma)\text{-diagonalizably stable}\}$) that controls the noiseless system.
}

\subsection{Spectral Controller}
{\color{black}
Online learning is an efficient framework for capturing the evolving pattern of a time-varying function sequence. However, this framework does not directly apply to linear policies of type $\ub=\Kb\xb$ due to the non-convex parameterization ($\xb_t$ is non-convex in terms of $\Kb$). To address this problem, \cite{agarwal2019online} proposed a disturbance-feedback controller (DFC), where the control is parameterized as a linear function of past noises from a specified time window, and both the state and control become linear functions of those coefficient matrices applied to the noises. In the work of \cite{brahmbhatt2025new}, it was shown that the dependency of the regret bound of DFC on the {\it stability margin} $\gamma$ (see Definition \ref{D: Diagonalizable Stability}) can be further improved by convolving past noises with the eigenvectors of a specific Hankel matrix $\Hb$ defined as:
\begin{equation}\label{Eq: Hankel matrix}
    [\Hb]_{ij} = \frac{(1-\gamma)^{i+j-1}}{i+j-1}.
\end{equation}
The definition of the corresponding spectral controller (SC) is given as follows:
\begin{definition}\label{D: Spectral Controller}
    A spectral controller $(\Mb_{1:h},\Hb)$ is defined by parameters $\Mb_{1:h} = \{\Mb_1,\ldots,\Mb_h\}$ for $h\geq 1$ and a Hankel matrix $\Hb\in \mathrm{R}^{m\times m}$ following Equation \eqref{Eq: Hankel matrix}. At round $t$, the control $\ub_t$ is determined as follows

    \begin{equation*}
        \ub_t = \sum_{i=1}^h \sigma_i^{1/4}\Mb_i \Wb_{t-1:t-m}\phi_i,
    \end{equation*}
    where $\{(\sigma_j,\phi_j)\}_{j=1}^h$ are the top $h$ eigenpairs of $\Hb$ and $\Wb_{t-1,t-m} = [\wb_{t-1},\ldots,\wb_{t-m}]$.    
\end{definition}
To capture the pattern of a time-varying cost sequence, $\Mb_{1:h}$ has to be updated over time, which leads to a time-varying SC: $\ub_t = \sum_{i=1}^h \sigma_i^{1/4}\Mb_i^t \Wb_{t-1:t-m}\phi_i$. By expanding the dynamics recursively, we can see for a time-varying SC, $\xb_{t+1}$ is a function of $\{\Mb^{t^{\prime}}_{1:h}\}_{t^{\prime}=1}^t$, which circumvents this parameterization from fitting into the framework of online learning since the number of parameters is growing over time. To tackle this problem, we resort to the the memory-less cost function \cite{brahmbhatt2025new} defined as follows:
\begin{equation}\label{eq:memoryless cost function}
    f_{i,t}(\Mb_{1:h}|\Ab,\Bb,\{\wb\}) = f_{i,t}\left(\xb_t(\Mb_{1:h}), \ub_t(\Mb_{1:h})\right),
\end{equation}
where $\xb_t(\Mb_{1:h})$ and $\ub_t(\Mb_{1:h})$ denote the resulting state and control at time $t$ when a fixed $\Mb_{1:h}$ is applied to the dynamical system.

}

\section{Main theoretical results}
{\color{black}

\begin{algorithm}[ht]
\caption{Distributed Online Spectral Control}
\label{alg:D-online spectral control}
\begin{algorithmic}[1]
    {\color{black}
        \STATE {\bfseries Require:} number of agents $n$, doubly stochastic matrix $\Pb\in \mathrm{R}^{n\times n}$, parameters $\gamma,\eta, m, h$, time horizon $T$ and the convex constraint set $\mathcal{K}$.

    \STATE Define\\
    $\mathcal{K} = \{\Mb_{1:h}\in\mathrm{R}^{d_2\times d_1\times h}|\norm{\xb_t(\Mb_{1:h})},\norm{\ub_t(\Mb_{1:h})}\leq \frac{3\kappa^3W}{\gamma}, \norm{\Mb_{1:h}}\leq \kappa^3\sqrt{\frac{2h}{\gamma}}\}$.
     
    \STATE {\bf Initialize:} $\forall i\in[n]$, randomly generate the same $\Mb_{1:h}^{i,1}\in\mathcal{K}$ for all $i$. 

    \STATE Compute $\{(\sigma_j,\phi_j)\}_{j=1}^h$, the top $h$ eigenpairs of the Hankel matrix $\Hb$.
  
    \FOR{$t=1,2,\ldots,T$}
        \FOR{$i=1,2,\ldots,n$}
            \STATE Determine the control $\ub_{i,t}$:
            \begin{equation*}
                \ub_{i,t} = \sum_{j=1}^h \sigma_j^{1/4}\Mb_j^{i,t}\Wb_{t-1:t-m}\phi_j.
            \end{equation*}
            \STATE Observe $\xb_{i,t+1}$ and compute $\wb_{t}=\xb_{i,t+1}-\Ab\xb_{i,t} - \Bb\ub_{i,t}$.                
            
            \STATE Compute $\widehat{\Mb}^{i,t+1}_{1:h} = \sum_{j=1}^n [\Pb]_{ji}\Mb^{j,t}_{1:h}-\eta\nabla_{\Mb} f_{i,t}(\Mb^{i,t}_{1:h}|\Ab,\Bb,\{\wb\})$.

            \STATE Compute $\Mb^{i,t+1}_{1:h} = \Pi_{\mathcal{K}}(\widehat{\Mb}^{i,t+1}_{1:h})$.
            
        \ENDFOR
   \ENDFOR
    
    }
\end{algorithmic}
\end{algorithm}

In this section, we introduce D-OSC, a distributed variant of online spectral controller which determines {\it local} controllers with information collected through {\it local} communication. At each round $t$, all agents jointly perform a distributed online gradient step over the memory-less cost functions $\{f_{i,t}(\Mb_{1:h}^{i,t}|\Ab,\Bb,\{\wb\})\}$. The update step consists two parts: 1) Each agent $i$ fist shares its current iterate $\Mb_{1:h}^{i,t}$ with its neighbors; 2) The weighted average iterate (based on the network topology) 
is then updated using the local gradient $\nabla_{\Mb_{1:h}}f_{i,t}(\Mb_{1:h}^{i,t}|\Ab,\Bb,\{\wb\})$. The updated policy parameters $\Mb_{1:h}^{i,t+1}$ will be used to decide the next local action $\ub_{i,t+1}$ for agent $i$. The learning process is summarized in Algorithm \ref{alg:D-online spectral control}. Later we show that for user-defined hyper-parameters such as $\eta$, $m$ and $h$, by properly selecting them as functions of $\gamma$ and $T$, the resulting individual regret bound is sub-linear in $T$ and enjoy better dependency on $\gamma$ than DFC.

For the technical analysis, we adhere to the standard assumptions used in \cite{agarwal2019online} and \cite{brahmbhatt2025new}. 
\begin{assumption}\label{A: Bounded B and noises}
    The system matrix $\Bb$ and the noise $\wb_t$ are bounded, i.e., $\norm{\Bb}\leq \kappa_B$ and $\norm{\wb_t}\leq W, \;\forall t$\footnote[3]{For ease of analysis, $\kappa$, $\kappa_B$ and $W$ are assumed to be greater than one without loss of generality.}.
\end{assumption}

\begin{assumption}\label{A: Gradient bound}
    The cost functions $f_{i,t}(\xb,\ub), \;\forall i,t$ are convex, and if $\norm{\xb}, \norm{\ub} \leq D$, the gradients are bounded as follows:
    \begin{equation*}
        \norm{\nabla_{\xb}f_{i,t}(\xb,\ub)}, \norm{\nabla_{\ub}f_{i,t}(\xb,\ub)} \leq GD.
    \end{equation*}
\end{assumption}

\begin{assumption}\label{A: zero is diagonally and strongly stable}
    The zero policy $K=0$ is $(\kappa,\gamma)$-diagonalizably stable.
\end{assumption}

The assumption that a zero policy is diagonalizably stable is mainly for the simplicity of the analysis. As shown in Section 8 of \cite{brahmbhatt2025new}, this assumption can be relaxed by using a precomputed $(\kappa,\gamma)$-strongly stable linear controller, which can be done by applying SDP relaxation described in \cite{cohen2018online}. The effect of the relaxed assumption on the main theoretical results boils down to an additional factor in terms of $\kappa$.

We now present our main theorem as follows:
\begin{theorem}\label{T: Main Theorem}
    Suppose Assumptions \ref{A: Bounded B and noises}, \ref{A: Gradient bound} and \ref{A: zero is diagonally and strongly stable} hold. Then by running Algorithm \ref{alg:D-online spectral control} with the hyper-parameters determined as follows:
    \begin{itemize}
        \item $\epsilon = \frac{1}{\sqrt{T}}$,

        \item $m = \lceil \frac{1}{\gamma}\log\left(\frac{8G\kappa_B\kappa^8W^2\sqrt{T}}{\gamma^3}\right)\rceil$,

        \item $h=2\log T\log\left(\frac{600G\kappa_B\kappa^8W^2\sqrt{mdT}}{\kappa^{5/2}}\log T\log^{1/4}\left(\frac{2}{\gamma}\right)\right)$,

        \item $\eta = \frac{\gamma^2}{mh^2\sqrt{T}}$,
    \end{itemize}
    we have the individual regret of any agent $j\in[n]$ being bounded as
    \begin{equation*}
        J_T^j(\mathcal{A}_1)-\min_{\Kb\in\mathcal{S}}J_T(\pi) = \Tilde{O}(\frac{n\sqrt{nT}}{(1-\beta)\gamma^3}),
    \end{equation*}
    where $\mathcal{S}$  denotes the set of $(\kappa,\gamma)$ diagonalizably stable linear controllers.
\end{theorem}
The analysis of Theorem \ref{T: Main Theorem} mainly consists of three parts: 1) The approximation part, where a diagonalizably stable linear controller is well-approximated by a a spectral controller; 2) The online optimization part, where part of the overall regret is parameterized as an online learning problem using memory-less cost functions ; 3) The gap between the cost induced by actual iterates $(\xb_{i,t}, \ub_{i,t})$ and that of the memory-less cost function. For the second part, through standard analysis of online convex optimization, the resulting bound is sublinear in $T$ and can be handled optimally when $\eta=\Theta(\frac{1}{\sqrt{T}})$. There exists a trade-off between the first and third terms: the approximation gap shrinks geometrically in terms of the memory length $m$, but the the third term increases when $m$ is larger. By properly selecting $m$ and $\eta$ as functions of the stability margin $\gamma$, we show that the dependency on $\gamma$, like the centralized case \cite{brahmbhatt2025new}, is improved in contrast to that derived based on the DFC parameterization. Also, the impact of the network topology is characterized in terms of the network size $n$ and $\beta$, the second largest singular value of $\Pb$, which is inversely proportional to the connectivity. We can see in the regret bound, when the network is more connected or comes with a smaller size, the resulting regret bound is tighter.

}

\section*{Conclusion}
In this work, we studied the distributed online control problem for LTI systems with adversarial disturbances and time-varying convex costs. Based on the spectral parameterization proposed in \cite{brahmbhatt2025new}, we reformulated this online control problem as a distributed online learning problem and proposed a learning approach called D-OSC, which achieves an individual regret bound of $O(\frac{n\sqrt{nT}\text{poly}(\log T)}{(1-\beta)\gamma^3})$. The resulting bound preserves the improved dependence on the stability margin inherited from the centralized spectral controller, while also characterizing the effect of the network size and connectivity.
Possible future directions include the extensions to unknown dynamics, partially observable systems and time-varying dynamics.

\bibliographystyle{IEEEtran}
\bibliography{references}

@inproceedings{agarwal2019online,
  title={Online control with adversarial disturbances},
  author={Agarwal, Naman and Bullins, Brian and Hazan, Elad and Kakade, Sham M and Singh, Karan},
  booktitle={International Conference on Machine Learning (ICML)},
  pages={154--165},
  year={2019},
}

@inproceedings{agarwal2019logarithmic,
  title={Logarithmic regret for online control},
  author={Agarwal, Naman and Hazan, Elad and Singh, Karan},
  booktitle={Advances in Neural Information Processing Systems (NeurIPS)},
  pages={10175--10184},
  year={2019}
}

@inproceedings{simchowitz2020improper,
  title={Improper learning for non-stochastic control},
  author={Simchowitz, Max and Singh, Karan and Hazan, Elad},
  booktitle={Conference on Learning Theory (COLT)},
  pages={3320--3436},
  year={2020},
  organization={PMLR}
}

@inproceedings{cohen2018online,
  title={Online Linear Quadratic Control},
  author={Cohen, Alon and Hasidim, Avinatan and Koren, Tomer and Lazic, Nevena and Mansour, Yishay and Talwar, Kunal},
  booktitle={International Conference on Machine Learning (ICML)},
  pages={1029--1038},
  year={2018}
}

@inproceedings{lale2022reinforcement,
  title={Reinforcement learning with fast stabilization in linear dynamical systems},
  author={Lale, Sahin and Azizzadenesheli, Kamyar and Hassibi, Babak and Anandkumar, Animashree},
  booktitle={International Conference on Artificial Intelligence and Statistics},
  pages={5354--5390},
  year={2022},
  organization={PMLR}
}

@inproceedings{zinkevich2003online,
  title={Online convex programming and generalized infinitesimal gradient ascent},
  author={Zinkevich, Martin},
  booktitle={Proceedings of the 20th international conference on machine learning (icml-03)},
  pages={928--936},
  year={2003}
}

@book{cesa2006prediction,
  title={Prediction, learning, and games},
  author={Cesa-Bianchi, Nicolo and Lugosi, G{\'a}bor},
  year={2006},
  publisher={Cambridge university press}
}

@article{hazan2016introduction,
  title={Introduction to Online Convex Optimization},
  author={Hazan, Elad},
  journal={Foundations and Trends in Optimization},
  volume={2},
  number={3-4},
  pages={157--325},
  year={2016},
  publisher={Now Publishers Inc.}
}

@inproceedings{hazan2017learning,
  title={Learning linear dynamical systems via spectral filtering},
  author={Hazan, Elad and Singh, Karan and Zhang, Cyril},
  booktitle={Advances in Neural Information Processing Systems (NeurIPS)},
  pages={6702--6712},
  year={2017}
}

@misc{
arora2018towards,
title={Towards Provable Control for Unknown Linear Dynamical Systems},
  author={Arora, Sanjeev and Hazan, Elad and Lee, Holden and Singh, Karan and Zhang, Cyril and Zhang, Yi},
year={2018},
}

@article{fattahi2019efficient,
  title={Efficient Learning of Distributed Linear-Quadratic Control Policies},
  author={Fattahi, Salar and Matni, Nikolai and Sojoudi, Somayeh},
  journal={SIAM Journal on Control and Optimization},
  volume={58},
  number={5},
  pages={2927--2951},
  year={2020},
}

@ARTICLE{5299181,  author={Y. {Cao} and W. {Ren}},  journal={IEEE Transactions on Systems, Man, and Cybernetics, Part B (Cybernetics)},   title={Optimal Linear-Consensus Algorithms: An LQR Perspective},   year={2010},  volume={40},  number={3},  pages={819-830},}

@ARTICLE{4626964,  author={F. {Borrelli} and T. {Keviczky}},  journal={IEEE Transactions on Automatic Control},   title={Distributed LQR Design for Identical Dynamically Decoupled Systems},   year={2008},  volume={53},  number={8},  pages={1901-1912},}

@INPROCEEDINGS{6862471,  author={A. {Mosebach} and J. {Lunze}},  booktitle={European Control Conference (ECC)},   title={Synchronization of autonomous agents by an optimal networked controller},   year={2014},  volume={},  number={},  pages={208-213},}

@ARTICLE{8736845,  author={J. {Jiao} and H. L. {Trentelman} and M. K. {Camlibel}},  journal={IEEE Transactions on Automatic Control},   title={A Suboptimality Approach to Distributed Linear Quadratic Optimal Control},   year={2020},  volume={65},  number={3},  pages={1218-1225},}

@inproceedings{alemzadeh2019distributed,
  title={Distributed q-learning for dynamically decoupled systems},
  author={Alemzadeh, Siavash and Mesbahi, Mehran},
  booktitle={American Control Conference (ACC)},
  pages={772--777},
  year={2019},
}

@inproceedings{furieri2020learning,
  title={Learning the globally optimal distributed LQ regulator},
  author={Furieri, Luca and Zheng, Yang and Kamgarpour, Maryam},
  booktitle={Learning for Dynamics and Control (L4DC)},
  pages={287--297},
  year={2020}
}

@inproceedings{cohen2019learning,
  title={Learning Linear-Quadratic Regulators Efficiently with only $\sqrt{T}$ Regret},
  author={Cohen, Alon and Koren, Tomer and Mansour, Yishay},
  booktitle={International Conference on Machine Learning (ICML)},
  pages={1300--1309},
  year={2019},
  organization={PMLR}
}

@inproceedings{dean2018regret,
  title={Regret bounds for robust adaptive control of the linear quadratic regulator},
  author={Dean, Sarah and Mania, Horia and Matni, Nikolai and Recht, Benjamin and Tu, Stephen},
  booktitle={International Conference on Neural Information Processing Systems (NeurIPS)},
  pages={4192--4201},
  year={2018}
}

@inproceedings{zhang2021regret,
  title={On the regret analysis of online LQR control with predictions},
  author={Zhang, Runyu and Li, Yingying and Li, Na},
  booktitle={American Control Conference (ACC)},
  pages={697--703},
  year={2021},
}

@inproceedings{chang2021distributed,
  title={Distributed online linear quadratic control for linear time-invariant systems},
  author={Chang, Ting-Jui and Shahrampour, Shahin},
  booktitle={American Control Conference (ACC)},
  pages={923--928},
  year={2021},
}

@inproceedings{cassel2020logarithmic,
  title={Logarithmic regret for learning linear quadratic regulators efficiently},
  author={Cassel, Asaf and Cohen, Alon and Koren, Tomer},
  booktitle={International Conference on Machine Learning (ICML)},
  pages={1328--1337},
  year={2020},
  organization={PMLR}
}

@inproceedings{simchowitz2020naive,
  title={Naive exploration is optimal for online lqr},
  author={Simchowitz, Max and Foster, Dylan},
  booktitle={International Conference on Machine Learning (ICML)},
  pages={8937--8948},
  year={2020},
  organization={PMLR}
}

@article{yu2020power,
  title={The Power of Predictions in Online Control},
  author={Yu, Chenkai and Shi, Guanya and Chung, Soon-Jo and Yue, Yisong and Wierman, Adam},
  journal={Advances in Neural Information Processing Systems (NeurIPS)},
  volume={33},
  year={2020}
}

@article{chang2023regret,
  title={Regret analysis of distributed online LQR control for unknown LTI systems},
  author={Chang, Ting-Jui and Shahrampour, Shahin},
  journal={IEEE Transactions on Automatic Control},
  year={2023},
  publisher={IEEE}
}

@article{karapetyan2022implications,
  title={Implications of Regret on Stability of Linear Dynamical Systems},
  author={Karapetyan, Aren and Tsiamis, Anastasios and Balta, Efe C and Iannelli, Andrea and Lygeros, John},
  journal={arXiv preprint arXiv:2211.07411},
  year={2022}
}

@article{nonhoff2023relation,
  title={On the relation between dynamic regret and closed-loop stability},
  author={Nonhoff, Marko and M{\"u}ller, Matthias A},
  journal={Systems \& Control Letters},
  volume={177},
  pages={105532},
  year={2023},
  publisher={Elsevier}
}

@inproceedings{burns2000techsat,
  title={TechSat 21: formation design, control, and simulation},
  author={Burns, Rich and McLaughlin, Craig A and Leitner, Jesse and Martin, Maurice},
  booktitle={2000 IEEE Aerospace Conference. Proceedings (Cat. No. 00TH8484)},
  volume={7},
  pages={19--25},
  year={2000},
  organization={IEEE}
}

@article{schaub2000spacecraft,
  title={Spacecraft formation flying control using mean orbit elements},
  author={Schaub, Hanspeter and Vadali, Srinivas R and Junkins, John L and Alfriend, Kyle T},
  journal={The Journal of the Astronautical Sciences},
  volume={48},
  pages={69--87},
  year={2000},
  publisher={Springer}
}

@inproceedings{buzogany1993automated,
  title={Automated control of aircraft in formation flight},
  author={Buzogany, L and Pachter, M and D'azzo, J},
  booktitle={Guidance, Navigation and Control Conference},
  pages={3852},
  year={1993}
}

@inproceedings{wolfe1996decentralized,
  title={Decentralized controllers for unmanned aerial vehicle formation flight},
  author={Wolfe, J and Chichka, D and Speyer, J},
  booktitle={Guidance, Navigation, and Control Conference},
  pages={3833},
  year={1996}
}

@inproceedings{yamaguchi1998cooperative,
  title={A cooperative hunting behavior by mobile robot troops},
  author={Yamaguchi, Hiroaki},
  booktitle={Proceedings. 1998 IEEE International Conference on Robotics and Automation (Cat. No. 98CH36146)},
  volume={4},
  pages={3204--3209},
  year={1998},
  organization={IEEE}
}

@article{yamaguchi2001distributed,
  title={A distributed control scheme for multiple robotic vehicles to make group formations},
  author={Yamaguchi, Hiroaki and Arai, Tamio and Beni, Gerardo},
  journal={Robotics and Autonomous systems},
  volume={36},
  number={4},
  pages={125--147},
  year={2001},
  publisher={Elsevier}
}

@article{yang2022collaborative,
  title={Collaborative navigation and manipulation of a cable-towed load by multiple quadrupedal robots},
  author={Yang, Chenyu and Sue, Guo Ning and Li, Zhongyu and Yang, Lizhi and Shen, Haotian and Chi, Yufeng and Rai, Akshara and Zeng, Jun and Sreenath, Koushil},
  journal={IEEE Robotics and Automation Letters},
  volume={7},
  number={4},
  pages={10041--10048},
  year={2022},
  publisher={IEEE}
}

@article{patel2019economic,
  title={Economic optimization of distributed embedded battery units for large-scale heating, ventilation, and air conditioning applications},
  author={Patel, Nishith R and Rawlings, James B and Ellis, Matthew J and Wenzel, Michael J and Turney, Robert D},
  journal={AIChE Journal},
  volume={65},
  number={7},
  pages={e16576},
  year={2019},
  publisher={Wiley Online Library}
}

@article{shin2023near,
  title={Near-optimal distributed linear-quadratic regulator for networked systems},
  author={Shin, Sungho and Lin, Yiheng and Qu, Guannan and Wierman, Adam and Anitescu, Mihai},
  journal={SIAM Journal on Control and Optimization},
  volume={61},
  number={3},
  pages={1113--1135},
  year={2023},
  publisher={SIAM}
}

@article{brahmbhatt2025new,
  title={A New Approach to Controlling Linear Dynamical Systems},
  author={Brahmbhatt, Anand and Buzaglo, Gon and Druchyna, Sofiia and Hazan, Elad},
  journal={arXiv preprint arXiv:2504.03952},
  year={2025}
}

@article{ackermann1972entwurf,
  title={Der entwurf linearer regelungssysteme im zustandsraum},
  author={Ackermann, J{\"u}rgen},
  journal={at-Automatisierungstechnik},
  number={7},
  pages={297--300},
  year={1972},
  publisher={Walter de Gruyter GmbH}
}

@article{sun2023optimal,
  title={Optimal rates for bandit nonstochastic control},
  author={Sun, Y Jennifer and Newman, Stephen and Hazan, Elad},
  journal={Advances in Neural Information Processing Systems},
  volume={36},
  pages={21908--21919},
  year={2023}
}

@inproceedings{suggala2024second,
  title={Second order methods for bandit optimization and control},
  author={Suggala, Arun and Sun, Y Jennifer and Netrapalli, Praneeth and Hazan, Elad},
  booktitle={The Thirty Seventh Annual Conference on Learning Theory},
  pages={4691--4763},
  year={2024},
  organization={PMLR}
}

@article{agarwal2023spectral,
  title={Spectral state space models},
  author={Agarwal, Naman and Suo, Daniel and Chen, Xinyi and Hazan, Elad},
  journal={arXiv preprint arXiv:2312.06837},
  year={2023}
}

@article{hazan2018spectral,
  title={Spectral filtering for general linear dynamical systems},
  author={Hazan, Elad and Lee, Holden and Singh, Karan and Zhang, Cyril and Zhang, Yi},
  journal={Advances in Neural Information Processing Systems},
  volume={31},
  year={2018}
}

@article{marsden2025dimension,
  title={Dimension-free Regret for Learning Asymmetric Linear Dynamical Systems},
  author={Marsden, Annie and Hazan, Elad},
  journal={arXiv e-prints},
  pages={arXiv--2502},
  year={2025}
}

\newpage
\onecolumn
\section{Supplementary}

\subsection{Complementary materials}

\begin{lemma}\label{L: Deviation between actual costs and reparameterize costs}
    Suppose Assumptions \ref{A: Bounded B and noises}, \ref{A: Gradient bound}
 and \ref{A: zero is diagonally and strongly stable} hold. Then by running Algorithm with $\eta = \frac{\gamma^2}{mh^2\sqrt{T}}$, we have $\forall i,j,t$:

    \begin{equation*}
        |f_{j,t}(\xb_t^{i,\mathcal{A}_1}, \ub_t^{i,\mathcal{A}_1}) - f_{j,t}(\Mb_{1:h}^{i,t}|\Ab,\Bb,\{\wb\})|\leq \frac{144 G^2\kappa_B^2\kappa^{10}W^4\sqrt{n}}{\gamma^3\sqrt{T}(1-\beta)}\log^{1/2}\left(\frac{2}{\gamma}\right),
    \end{equation*}
    for $T \geq \left(\frac{36G\kappa_B^2\kappa^4W^2\sqrt{n}}{\gamma(1-\beta)}\log^{1/2}\left(\frac{2}{\gamma}\right)\right)^2$.    
\end{lemma}
\begin{proof}
    Based on how each agent's iterate is updated, for each $k\in[t-1]$ and $h^{\prime}\in[h]$, we have
    \begin{equation}\label{Eq2: L Deviation between actual costs and reparameterize costs}
    \begin{split}
        &\norm{\Mb_{h^{\prime}}^{i,t} - \Mb_{h^{\prime}}^{i,t-k}}\\
        \leq &\norm{\Mb_{1:h}^{i,t} - \Mb_{1:h}^{i,t-k}}\leq \sum_{s=t-k+1}^t \norm{\Mb_{1:h}^{i,s} - \Mb_{1:h}^{i,s-1}}\\
        \leq &\sum_{s=t-k+1}^t \left[\norm{\Mb_{1:h}^{i,s} - \widehat{\Mb}_{1:h}^{i,s}} + \norm{\widehat{\Mb}_{1:h}^{i,s} - \Mb_{1:h}^{i,s-1}}\right]\\
        \leq &\sum_{s=t-k+1}^t\left[\norm{r_{i,t}} + \norm{\sum_{j=1}^n\left[\Pb_{ij}\left(\Mb^{j,t-1}_{1:h} - \Mb^{t-1}_{1:h} + \Mb^{t-1}_{1:h} - \Mb^{i,t-1}_{1:h}\right)\right] - \eta\nabla f_{i,t-1}(\Mb_{1:h}^{i,t-1})|\Ab,\Bb,\{\wb\}}\right]\\
        \leq &\sum_{s=t-k+1}^t \left[\eta \frac{12G\kappa_B\kappa^5W^2\sqrt{m}h}{\gamma^2}\log^{1/4}\left(\frac{2}{\gamma}\right) + \frac{24\eta G\kappa_B\kappa^5W^2\sqrt{m}h\sqrt{n}}{\gamma^2(1-\beta)}\log^{1/4}\left(\frac{2}{\gamma}\right)\right]\\
        \leq &k\frac{36\eta G\kappa_B\kappa^5W^2\sqrt{m}h\sqrt{n}}{\gamma^2(1-\beta)}\log^{1/4}\left(\frac{2}{\gamma}\right).
    \end{split}
    \end{equation}
    From the algorithm, we can see $\ub_t^{i,\mathcal{A}_1}$ solely depends on $\Mb_{1:h}^{i,t}$, which imples that $\ub_t^{i,\mathcal{A}_1} = \ub_t(\Mb_{1:h}^{i,t})$. 

    As for the deviation between $\xb_t^{i,\mathcal{A}_1}$ and $\xb_t(\Mb_{1:h}^{i,t})$, first we observe these two terms can be written as follows
    \begin{equation}\label{Eq3: L Deviation between actual costs and reparameterize costs}
    \begin{split}
        \xb_t(\Mb_{1:h}^{i,t}) &= \sum_{j=1}^t \Ab^{j-1}\wb_{t-j} + \sum_{j=1}^t \Ab^{j-1}\Bb\sum_{h^{\prime}=1}^h \sigma_{h^{\prime}}^{1/4}\Mb_{h^{\prime}}^{i,t} \Wb_{t-j-1:t-j-m}\phi_{h^{\prime}},\\
        \xb_t^{i,\mathcal{A}_1} &= \sum_{j=1}^t \Ab^{j-1}\wb_{t-j} + \sum_{j=1}^t \Ab^{j-1}\Bb\sum_{h^{\prime}=1}^h \sigma_{h^{\prime}}^{1/4}\Mb_{h^{\prime}}^{i,t-j} \Wb_{t-j-1:t-j-m}\phi_{h^{\prime}},
    \end{split}    
    \end{equation}
    upon which we have
    \begin{equation}\label{Eq3: L Deviation between actual costs and reparameterize costs}
    \begin{split}
        \norm{\xb_t^{i,\mathcal{A}_1} - \xb_t(\Mb_{1:h}^{i,t})} &\leq \sum_{j=1}^t \norm{\Ab^{j-1}}\norm{\Bb}\sum_{h^{\prime}=1}^h |\sigma_{h^{\prime}}^{1/4}|\norm{\Mb_{h^{\prime}}^{i,t-j} - \Mb_{h^{\prime}}^{i,t}} \norm{\Wb_{t-j-1:t-j-m}}\\
        &\leq \sum_{j=1}^t \kappa^2 (1-\gamma)^{j-1}\kappa_B\sum_{h^{\prime}=1}^h \left[\log^{1/4}\left(\frac{2}{\gamma}\right)j\frac{36\eta G\kappa_B\kappa^5W^2\sqrt{m}h\sqrt{n}}{\gamma^2(1-\beta)}\log^{1/4}\left(\frac{2}{\gamma}\right) W\sqrt{m}\right]\\
        &\leq \frac{36\eta G\kappa_B^2\kappa^7W^3mh^2\sqrt{n}}{\gamma^2(1-\beta)}\log^{1/2}\left(\frac{2}{\gamma}\right)\sum_{j=1}^t j(1-\gamma)^{j-1}\leq \frac{36\eta G\kappa_B^2\kappa^7W^3mh^2\sqrt{n}}{\gamma^4(1-\beta)}\log^{1/2}\left(\frac{2}{\gamma}\right).
    \end{split}        
    \end{equation}

    Based on the definition of $\mathcal{K}$ and {\color{black} the selection of $T$}, we have both $\norm{\xb_t^{i,\mathcal{A}_1}},\;\norm{\xb_t(\Mb_{1:h}^{i,t})}\leq \frac{4\kappa^3W}{\gamma}$. Then based on Assumption \ref{A: Gradient bound}, we have $\forall j\in[n]$
    \begin{equation}\label{Eq4: L Deviation between actual costs and reparameterize costs}
    \begin{split}
        &|f_{j,t}(\xb_t^{i,\mathcal{A}_1}, \ub_t^{i,\mathcal{A}_1}) - f_{j,t}(\Mb_{1:h}^{i,t}|\Ab,\Bb,\{\wb\})|\\
        = &|f_{j,t}(\xb_t^{i,\mathcal{A}_1}, \ub_t^{i,\mathcal{A}_1}) - f_{j,t}(\xb_t(\Mb_{1:h}^{i,t}), \ub_t(\Mb_{1:h}^{i,t}))|\\
        \leq &\frac{4G\kappa^3W}{\gamma}\norm{\xb_t^{i,\mathcal{A}_1} - \xb_t(\Mb_{1:h}^{i,t})}\\
        \leq &\frac{144\eta G^2\kappa_B^2\kappa^{10}W^4mh^2\sqrt{n}}{\gamma^5(1-\beta)}\log^{1/2}\left(\frac{2}{\gamma}\right)\\
        \leq &\frac{144 G^2\kappa_B^2\kappa^{10}W^4\sqrt{n}}{\gamma^3\sqrt{T}(1-\beta)}\log^{1/2}\left(\frac{2}{\gamma}\right)
    \end{split}         
    \end{equation}

\end{proof}

\begin{lemma}\label{L: Deviation between costs of normal and truncated linear policies}
    Let a linear policy $\Kb\in\Sc$. Then, for $m\geq \frac{1}{\gamma} \log \left(\frac{8G\kappa_B \kappa^8W^2}{\epsilon \gamma^3}\right)$ and $\epsilon \in (0,1)$,
    \begin{equation*}
        \sum_{t=1}^T|f_{i,t}(\xb^{\Kb,m}_t,\ub^{\Kb,m}_t)-f_{i,t}(\xb^{\Kb}_t,\ub^{\Kb}_t)|\leq \frac{\epsilon}{2}T,~~~~\norm{\xb_t^{\Kb,m}}, \norm{\ub_t^{\Kb,m}}\leq \frac{2\kappa^3W}{\gamma},
    \end{equation*}
    $\forall i\in[n]$.
\end{lemma}

\begin{lemma}\label{L: Deviation between costs of SC and truncated linear policies}
    For every open loop optimal controller $\pi^{\textbf{OLOC}}_{\Kb,m}$ such that $\Kb\in\Sc$ and $\norm{\xb_t^{\Kb,m}}, \norm{\ub_t^{\Kb,m}}\leq \frac{2\kappa^3W}{\gamma}$, there exists a spectral controller $\pi^{\textbf{SC}}_{h,m,\gamma,\Mb}$ with $\Mb\in\mathcal{K}$ such that:
    \begin{equation*}
        \sum_{t=1}^T|f_{i,t}(\xb^{\Mb}_t,\ub^{\Mb}_t)-f_{i,t}(\xb^{\Kb,m}_t,\ub^{\Kb,m}_t)|\leq \frac{\epsilon}{2}T
    \end{equation*}
    for any $\epsilon\in(0,1)$ and $h\geq 2\log T\log\left(\frac{600G\kappa_B\kappa^8W^2\sqrt{md}}{\epsilon \gamma^{5/2}}\log T\log^{1/4}\left(\frac{2}{\gamma}\right)\right)$.
\end{lemma}

\begin{lemma}\label{L: Bounded states and actions}
    For any $\Kb\in \Sc$, the corresponding states $\xb_t^{\Kb}$ and control inputs $\ub_t^{\Kb}$ are bounded as follows,
    \begin{equation*}
        \norm{\xb^{\Kb}_t}\leq \frac{\kappa^2 W}{\gamma}, \norm{\ub^{\Kb}_t}\leq \frac{\kappa^3 W}{\gamma}.
    \end{equation*}
\end{lemma}

\begin{lemma}\label{L: Lipschitz constant of the reparameterized costs}
    For any $\Mb_{1:h},\;\Mb^{\prime}_{1:h} \in \mathcal{K}$, it holds that $\forall t,i$,
    \begin{equation*}
        |f_{i,t}(\Mb_{1:h}|\Ab,\Bb,\{\wb\}) - f_{i,t}(\Mb_{1:h}^{\prime}|\Ab,\Bb,\{\wb\})|\leq \frac{6G\kappa_B\kappa^5W^2\sqrt{m}h}{\gamma^2}\log^{1/4}\left(\frac{2}{\gamma}\right)\norm{\Mb_{1:h} - \Mb^{\prime}_{1:h}}.        
    \end{equation*}
\end{lemma}

\begin{lemma}\label{L: Bound of the eigenvalues of the Hankel matrix}
    Let $\sigma_j$ be the $j^{th}$ top singular value of $\Hb_m$. Then, for all $T\geq 10$, we have
    \begin{equation*}
        \sigma_j\leq 156800\log \left(\frac{2}{\gamma}\right)\cdot \exp\left(-\frac{\pi^2j}{4\log T}\right)\leq \frac{1}{2}\log\left(\frac{2}{\gamma}\right).
    \end{equation*}
\end{lemma}

\begin{lemma}\label{L: Bound on the inner product between a geometric vector and the eigenvector of Hankel matrix}
    Let $\{\phi_j\}$ be the eigenvectors of $\Hb_m$. Then for all $j\in [m],\alpha \in [0,1-\gamma]$ and $\gamma\leq 2/3$,
    \begin{equation*}
        |\mu_{\alpha}^{\top}\phi_j|\leq \sqrt{\frac{2}{\gamma}}\sigma^{1/4}\leq \frac{30}{\sqrt{\gamma}}\log^{1/4}\left(\frac{2}{\gamma}\right)\exp\left(-\frac{\pi^2j}{16\log T}\right). 
    \end{equation*}
\end{lemma}

\begin{lemma}\label{L: Consecutive distance (SC)}
    Let Algorithm \ref{alg:D-online control} run with step size $\eta>0$ and define $\Mb^t_{1:h} = \frac{1}{m}\sum_{i=1}^n\Mb^{i,t}_{1:h}$. Under Assumptions \ref{A1}, \ref{A2} and \ref{A3}, we have that $\forall i \in [n]$    
    \begin{equation*}
        \norm{\Mb^{t}_{1:h} - \Mb^{i,t}_{1:h}}_F\leq \frac{12\eta G\kappa_B\kappa^5W^2\sqrt{m}h\sqrt{n}}{\gamma^2(1-\beta)}\log^{1/4}\left(\frac{2}{\gamma}\right)
    \end{equation*}
    and
    \begin{equation*}
        \norm{\Mb^t_{1:h} - \widehat{\Mb}^{i,t+1}_{1:h}}_F\leq \frac{12\eta G\kappa_B\kappa^5W^2\sqrt{m}h\sqrt{n}}{\gamma^2(1-\beta)}\log^{1/4}\left(\frac{2}{\gamma}\right).
    \end{equation*}
\end{lemma}
\begin{proof}
    In this proof, we consider the vectorized versions of $\Mb^{i,t}_{1:h}$, $\Mb^{t}_{1:h}$ and $\nabla f_{i,t}(\Mb^{i,t}_{1:h}|\Ab,\Bb,\{\wb\})$. For the ease of the analysis, we define the following matrices
    \begin{equation}\label{Eq1: Consecutive distance}
    \begin{split}
        \mathbb{M}_t &= [\Mb^{1,t}_{1:h},\ldots,\Mb^{n,t}_{1:h}]\\
        \widehat{\mathbb{M}}_t &= [\widehat{\Mb}^{1,t}_{1:h},\ldots,\widehat{\Mb}^{n,t}_{1:h}]\\
        \Gb_t &= [\nabla f_{1,t}(\Mb^{1,t}_{1:h}|\Ab,\Bb,\{\wb\}),\ldots,\nabla f_{n,t}(\Mb^{n,t}_{1:h}|\Ab,\Bb,\{\wb\})]\\
        \Rb_t &= [r_{1,t},\ldots,r_{n,t}],
    \end{split}
    \end{equation}
    where $r_{i,t} = \widehat{\Mb}^{i,t}_{1:h} - \Mb^{i,t}_{1:h}$. Based on these notations, the update can be expressed as $\mathbb{M}_t = \mathbb{M}_{t-1}\Pb - \eta \Gb_{t-1} - \Rb_t$.

    Expanding the update recursively, we get
    \begin{equation*}
        \mathbb{M}_t = \mathbb{M}_1\Pb^{t-1} - \sum_{l=1}^{t-1}\eta \Gb_{t-l}\Pb^{l-1} - \sum_{l=0}^{t-2} \Rb_{t-l}\Pb^{l}.
    \end{equation*}
    As $\Pb$ is doubly stochastic, we have $\Pb^k\mathbf{1}=\mathbf{1}$ for all $k\geq 1$. Based on the geometric mixing bound of $\Pb$, we have
    \begin{equation*}
    \begin{split}
        &\norm{\Mb_{t} - \Mb_{i,t}}\\
        = &\norm{\mathbb{M}_t(\frac{1}{n}\mathbf{1}-\eb_i)}\\
        \leq &\norm{\mathbb{M}_{1}\frac{1}{n}\mathbf{1}-\mathbb{M}_{1}[\Pb^{(t-1)}]_{:,i}} + \eta\sum_{l=1}^{t-1}\norm{\Gb_{t-l}(\frac{1}{n}\mathbf{1}-[\Pb^{l-1}]_{:,i})}
        +\sum_{l=0}^{t-2}\norm{\Rb_{t-l}(\frac{1}{n}\mathbf{1}-[\Pb^{l}]_{:,i})}\\
        \leq & \eta \frac{6G\kappa_B\kappa^5W^2\sqrt{m}h}{\gamma^2}\log^{1/4}\left(\frac{2}{\gamma}\right)\sum_{l=1}^{t-1}\sqrt{n}\beta^{l-1} + \eta \frac{6G\kappa_B\kappa^5W^2\sqrt{m}h}{\gamma^2}\log^{1/4}\left(\frac{2}{\gamma}\right) \sum_{l=0}^{t-2}\sqrt{n}\beta^l\\
        \leq &\frac{12\eta G\kappa_B\kappa^5W^2\sqrt{m}h\sqrt{n}}{\gamma^2(1-\beta)}\log^{1/4}\left(\frac{2}{\gamma}\right), 
    \end{split}    
    \end{equation*}
    where the second inequality is due to Lemma \ref{L: Lipschitz constant of the reparameterized costs} and the fact that $\norm{\mathbb{M}_{1}-\mathbb{M}_{1}[\Pb^{(t-1)}]_{:,i}}=0$ by the identical initialization. Also based on Lemma \ref{L: Lipschitz constant of the reparameterized costs} we have
    \begin{equation*}
    \begin{split}
        \norm{r_{i,t}} &= \norm{\widehat{\Mb}^{i,t}_{1:h} - \Mb^{i,t}_{1:h}}\\
        &\leq \norm{ \big(\sum_j [\Pb^{}]_{ji}\Mb^{j,t-1}_{1:h}-\eta\nabla f_{i,t-1}(\Mb^{i,t-1}_{1:h})\big) - \sum_j [\Pb^{}]_{ji}\Mb^{j,t-1}_{1:h}}\\
        &\leq \eta \frac{6G\kappa_B\kappa^5W^2\sqrt{m}h}{\gamma^2}\log^{1/4}\left(\frac{2}{\gamma}\right),
    \end{split}
    \end{equation*}
    where the first inequality is due to projection on convex sets.\\\\
    By the same token, 
    \begin{equation*}
    \begin{split}
        &\norm{\Mb^t_{1:h} - \widehat{\Mb}^{i,t+1}_{1:h}}=\norm{\frac{1}{m}\mathbb{M}_t \mathbf{1}-(\mathbb{M}_t\Pb-\eta \Gb_t)\eb_i}\\
        =&\norm{\mathbb{M}_t(\frac{1}{m}\mathbf{1}-\Pb\eb_i)+\eta \Gb_t\eb_i}\\
        \leq &\norm{\Mb_{1} - \mathbb{M}_{1}[\Pb^{t}]_{:,i}} + \eta\sum_{l=1}^{t-1}\norm{\Gb_{t-l}(\frac{1}{m}\mathbf{1}-[\Pb^{l}]_{:,i})}\\
        +&\sum_{l=0}^{t-2}\norm{R_{t-l}(\frac{1}{m}\mathbf{1}-[\Pb^{l+1}]_{:,i})} + \norm{\eta \nabla f_{i,t}(\Mb_{i,t})}\\
        \leq &\eta \frac{6G\kappa_B\kappa^5W^2\sqrt{m}h}{\gamma^2}\log^{1/4}\left(\frac{2}{\gamma}\right)\\
        + &\eta \frac{6G\kappa_B\kappa^5W^2\sqrt{m}h}{\gamma^2}\log^{1/4}\left(\frac{2}{\gamma}\right)\left[\sum_{l=1}^{t-1}\sqrt{m}\beta^l + \sum_{l=0}^{t-2}\sqrt{m}\beta^{l+1}\right]\\
        \leq &\frac{12\eta G\kappa_B\kappa^5W^2\sqrt{m}h\sqrt{n}}{\gamma^2(1-\beta)}\log^{1/4}\left(\frac{2}{\gamma}\right).
    \end{split}    
    \end{equation*}
        
\end{proof}

\begin{theorem}\label{T: Distributed regret bound}
    Suppose Assumptions \ref{A: Bounded B and noises}, \ref{A: Gradient bound} and \ref{A: zero is diagonally and strongly stable} hold. Then by running Algorithm, we have for any $\Mb_{1:h}\in\mathcal{K}$,
    \begin{equation*}
    \begin{split}
        &\sum_{t=1}^T\sum_{i=1}^n \left[f_{i,t}(\Mb_{1:h}^{j,t}|\Ab,\Bb,\{\wb\}) - f_{i,t}(\Mb_{1:h}|\Ab,\Bb,\{\wb\})\right]\\
        \leq &\frac{n}{2\eta}\norm{\Mb_{1:h}^1 - \Mb_{1:h}}^2 + T\left(\frac{8\eta n\sqrt{n}}{(1-\beta)} + 2\eta\right)\left(\frac{6G\kappa_B\kappa^5W^2\sqrt{m}h}{\gamma^2}\log^{1/4}\left(\frac{2}{\gamma}\right)\right)^2.
    \end{split}
    \end{equation*}
\end{theorem}
\begin{proof}
    Following the notation defined in Equation \eqref{Eq1: Consecutive distance}, we have the following equations for the iterate $\Mb^t_{1:h}$:
    \begin{equation}\label{Eq1: T Distributed regret bound}
    \begin{split}
        \Mb_{1:h}^{t+1}&=\frac{1}{n}\mathbb{M}_{t+1}\mathbf{1}=\frac{1}{n}(\mathbb{M}_t\Pb-\eta\Gb_t-\Rb_{t+1})\mathbf{1}\\
        &=\frac{1}{n}\mathbb{M}_t\mathbf{1} - \frac{\eta}{n}\Gb_t\mathbf{1} - \frac{1}{n}\Rb_{t+1}\mathbf{1}=\Mb_{1:h}^{t}-\frac{\eta}{n}\sum_{i=1}^n\nabla f_{i,t}(\Mb_{1:h}^{i,t}|\Ab,\Bb,\{\wb\})-\frac{1}{n}\sum_{i=1}^n r^i_{t+1}.
    \end{split}
    \end{equation}
    Then for any $\Mb_{1:h}\in\mathcal{K}$, we have
    \begin{equation}\label{Eq2: T Distributed regret bound}
    \begin{split}
        \norm{\Mb_{1:h}^{t+1} - \Mb_{1:h}}^2 &= \norm{\Mb_{1:h}^{t} - \Mb_{1:h}}^2 + \frac{1}{n^2}\norm{\sum_{i=1}^n\left(-r_{t+1}^i - \eta\nabla f_{i,t}(\Mb_{1:h}^{i,t}|\Ab,\Bb,\{\wb\})\right)}^2\\
        &-\frac{2\eta}{n}\sum_{i=1}^n\langle \Mb_{1:h}^{t} - \Mb_{1:h}, \nabla f_{i,t}(\Mb_{1:h}^{i,t}|\Ab,\Bb,\{\wb\})\rangle - \frac{2}{n}\sum_{i=1}^n\langle \Mb_{1:h}^{t} - \Mb_{1:h}, r_{t+1}^i\rangle.
    \end{split}
    \end{equation}
    Based on Lemma \ref{L: Lipschitz constant of the reparameterized costs}, we have
    \begin{equation}\label{Eq3: T Distributed regret bound}
    \begin{split}
        \frac{1}{n^2}\norm{\sum_{i=1}^n\left(-r_{t+1}^i - \eta\nabla f_{i,t}(\Mb_{1:h}^{i,t}|\Ab,\Bb,\{\wb\})\right)}^2
        &\leq \frac{1}{n^2}\left[\sum_{i=1}^n\left(\norm{r_{t+1}^i} + \eta\norm{\nabla f_{i,t}(\Mb_{1:h}^{i,t}|\Ab,\Bb,\{\wb\})}\right)\right]^2\\
        &\leq \left( \eta \frac{12G\kappa_B\kappa^5W^2\sqrt{m}h}{\gamma^2}\log^{1/4}\left(\frac{2}{\gamma}\right) \right)^2.
    \end{split}    
    \end{equation}
    As $f_{i,t}(\Mb_{1:h}|\Ab,\Bb,\{\wb\})$ is convex in terms of $\Mb_{1:h}$, we have 
    \begin{equation}\label{Eq4: T Distributed regret bound}
    \begin{split}
        &\langle \Mb_{1:h}-\Mb^t_{1:h}, \nabla f_{i,t}(\Mb_{1:h}^{i,t}|\Ab,\Bb,\{\wb\}) \rangle \\
        = &\langle \Mb_{1:h} - \Mb^{i,t}_{1:h} + \Mb^{i,t}_{1:h} - \Mb^t_{1:h}, \nabla f_{i,t}(\Mb_{1:h}^{i,t}|\Ab,\Bb,\{\wb\})\rangle\\
        \leq &f_{i,t}(\Mb_{1:h}|\Ab,\Bb,\{\wb\}) - f_{i,t}(\Mb_{1:h}^{i,t}|\Ab,\Bb,\{\wb\}) + \norm{\Mb^{i,t}_{1:h} - \Mb^t_{1:h}}\norm{\nabla f_{i,t}(\Mb_{1:h}^{i,t}|\Ab,\Bb,\{\wb\})}\\
        \leq &f_{i,t}(\Mb_{1:h}|\Ab,\Bb,\{\wb\}) - f_{i,t}(\Mb_{1:h}^{i,t}|\Ab,\Bb,\{\wb\}) + \norm{\Mb^{i,t}_{1:h} - \Mb^t_{1:h}}\frac{6G\kappa_B\kappa^5W^2\sqrt{m}h}{\gamma^2}\log^{1/4}\left(\frac{2}{\gamma}\right)
    \end{split}        
    \end{equation}
    For the term $\langle \Mb_{1:h}^t-\Mb_{1:,h}, -r_{t+1}^i\rangle$, its upper bound is shown as follows
    \begin{equation}\label{Eq5: T Distributed regret bound}
    \begin{split}
        \langle \Mb_{1:h}^t-\Mb_{1:,h}, -r_{t+1}^i\rangle
        &= \langle \Mb_{1:h}^t-\widehat{\Mb}^{i,t+1}_{1:,h}, -r_{t+1}^i\rangle + \langle \widehat{\Mb}^{i,t+1}_{1:,h}-\Mb_{1:,h}, -r_{t+1}^i\rangle\\
        &\leq \norm{\Mb_{1:h}^t-\widehat{\Mb}^{i,t+1}_{1:,h}}\norm{r_{t+1}^i}\\
        &\leq \norm{\Mb_{1:h}^t-\widehat{\Mb}^{i,t+1}_{1:,h}}\eta \frac{6G\kappa_B\kappa^5W^2\sqrt{m}h}{\gamma^2}\log^{1/4}\left(\frac{2}{\gamma}\right),     
    \end{split}    
    \end{equation}
    where the second inequality is due to the fact that $\langle \widehat{\Mb}^{i,t+1}_{1:,h}-\Mb_{1:,h}, -r_{t+1}^i\rangle$ is non-positive because of the property of a projection operator on a convex set, and the last inequality is based on Equation. 

    Substituting Equations \eqref{Eq3: T Distributed regret bound}, \eqref{Eq4: T Distributed regret bound} and \eqref{Eq5: T Distributed regret bound} into Equation \eqref{Eq2: T Distributed regret bound}, we can get
    \begin{equation}\label{Eq6: T Distributed regret bound}
    \begin{split}
        &\norm{\Mb_{1:h}^{t+1} - \Mb_{1:h}}^2 - \norm{\Mb_{1:h}^{t} - \Mb_{1:h}}^2\\
        \leq &\left( \eta \frac{12G\kappa_B\kappa^5W^2\sqrt{m}h}{\gamma^2}\log^{1/4}\left(\frac{2}{\gamma}\right) \right)^2 + \frac{2}{n}\sum_{i=1}^n \left[\norm{\Mb_{1:h}^t-\widehat{\Mb}^{i,t+1}_{1:,h}}\eta \frac{6G\kappa_B\kappa^5W^2\sqrt{m}h}{\gamma^2}\log^{1/4}\left(\frac{2}{\gamma}\right)\right]\\
        + &\frac{2\eta}{n}\sum_{i=1}^n\left[f_{i,t}(\Mb_{1:h}|\Ab,\Bb,\{\wb\}) - f_{i,t}(\Mb_{1:h}^{i,t}|\Ab,\Bb,\{\wb\}) + \norm{\Mb^{i,t}_{1:h} - \Mb^t_{1:h}}\frac{6G\kappa_B\kappa^5W^2\sqrt{m}h}{\gamma^2}\log^{1/4}\left(\frac{2}{\gamma}\right)\right]
        \end{split}
    \end{equation}

    Due to the convexity of $f_{i,t}(\cdot)$ and Lemma \ref{L: Lipschitz constant of the reparameterized costs}, we have the following inequality: 
    \begin{equation}\label{Eq7: T Distributed regret bound}
    \begin{split}
        &f_{i,t}(\Mb_{1:h}^{j,t}|\Ab,\Bb,\{\wb\}) - f_{i,t}(\Mb_{1:h}^{i,t}|\Ab,\Bb,\{\wb\})\\
        \leq &\norm{\nabla f_{i,t}(\Mb_{1:h}^{j,t}|\Ab,\Bb,\{\wb\})}\norm{\Mb_{1:h}^{j,t} - \Mb_{1:h}^{i,t}}\\
        \leq &\frac{6G\kappa_B\kappa^5W^2\sqrt{m}h}{\gamma^2}\log^{1/4}\left(\frac{2}{\gamma}\right)\left(\norm{\Mb_{1:h}^{j,t} - \Mb_{1:h}^{t}} + \norm{\Mb_{1:h}^{t} - \Mb_{1:h}^{i,t}}\right)
    \end{split}            
    \end{equation}

    By adding and subtracting $f_{i,t}(\Mb_{1:h}^{j,t}|\Ab,\Bb,\{\wb\})$ on the left hand side of Equation \eqref{Eq6: T Distributed regret bound} and rearranging it, we have the following inequality based on Equation \eqref{Eq7: T Distributed regret bound}: 
    \begin{equation}\label{Eq8: T Distributed regret bound}
    \begin{split}
        &\sum_{i=1}^n \left[f_{i,t}(\Mb_{1:h}^{j,t}|\Ab,\Bb,\{\wb\}) - f_{i,t}(\Mb_{1:h}|\Ab,\Bb,\{\wb\})\right]\\
        \leq &\frac{n}{2\eta}\left(\norm{\Mb_{1:h}^{t} - \Mb_{1:h}}^2 - \norm{\Mb_{1:h}^{t+1} - \Mb_{1:h}}^2\right) + \sum_{i=1}^n \left[\norm{\Mb_{1:h}^t-\widehat{\Mb}^{i,t+1}_{1:,h}} \frac{6G\kappa_B\kappa^5W^2\sqrt{m}h}{\gamma^2}\log^{1/4}\left(\frac{2}{\gamma}\right)\right]\\
        + &\sum_{i=1}^n\left[\frac{6G\kappa_B\kappa^5W^2\sqrt{m}h}{\gamma^2}\log^{1/4}\left(\frac{2}{\gamma}\right)\left(\norm{\Mb_{1:h}^{j,t} - \Mb_{1:h}^{t}} + 2\norm{\Mb_{1:h}^{t} - \Mb_{1:h}^{i,t}}\right)\right]\\
        + &\frac{n}{2\eta}\left( \eta \frac{12G\kappa_B\kappa^5W^2\sqrt{m}h}{\gamma^2}\log^{1/4}\left(\frac{2}{\gamma}\right) \right)^2\\
        \leq & \frac{n}{2\eta}\left(\norm{\Mb_{1:h}^{t} - \Mb_{1:h}}^2 - \norm{\Mb_{1:h}^{t+1} - \Mb_{1:h}}^2\right) + \frac{n}{2\eta}\left( \eta \frac{12G\kappa_B\kappa^5W^2\sqrt{m}h}{\gamma^2}\log^{1/4}\left(\frac{2}{\gamma}\right) \right)^2\\
        + &\frac{8\eta n\sqrt{n}}{(1-\beta)}\left(\frac{6G\kappa_B\kappa^5W^2\sqrt{m}h}{\gamma^2}\log^{1/4}\left(\frac{2}{\gamma}\right)\right)^2,
    \end{split} 
    \end{equation}
    where the second inequality is based on Lemma \ref{L: Consecutive distance (SC)}. Then by summing Equation \eqref{Eq8: T Distributed regret bound} over $t=1,\ldots,T$, the theorem is proved.

\end{proof}

\begin{theorem}\label{T: Main Theorem}
    Suppose Assumptions \ref{A: Bounded B and noises}, \ref{A: Gradient bound} and \ref{A: zero is diagonally and strongly stable} hold. Then by running Algorithm, we have the individual regret bound of the following order
    \begin{equation*}
        \textbf{Regret} = \Tilde{O}(\frac{n\sqrt{nT}}{(1-\beta)\gamma^3}),
    \end{equation*}
    with the hyperparameters determined as follows:
    \begin{itemize}
        \item $\epsilon = \frac{1}{\sqrt{T}}$

        \item $m = \lceil \frac{1}{\gamma}\log\left(\frac{8G\kappa_B\kappa^8W^2\sqrt{T}}{\gamma^3}\right)\rceil$

        \item $h=2\log T\log\left(\frac{600G\kappa_B\kappa^8W^2\sqrt{mdT}}{\kappa^{5/2}}\log T\log^{1/4}\left(\frac{2}{\gamma}\right)\right)$

        \item $\eta = \frac{\gamma^2}{mh^2\sqrt{T}}$
    \end{itemize}
\end{theorem}

\begin{proof}{Proof of Main Theorem}
    Consider $\Kb^*$ as the optimal linear policy in hindsight. Then, the individual regret is decomposed as follows:
    \begin{equation}\label{Eq1: T Main Theorem}
    \begin{split}
        &J_T^j(\mathcal{A}_1) - \min_{\Kb}J_T(\Kb)\\
        =&\sum_{t=1}^T\sum_{i=1}^n\left[f_{i,t}(\xb_t^{j,\mathcal{A}_1},\ub_t^{j,\mathcal{A}_1}) - f_{i,t}(\xb_t^{\Kb^*},\ub_t^{\Kb^*})\right]\\
        =&\sum_{t=1}^T\sum_{i=1}^n\left[f_{i,t}(\xb_t^{j,\mathcal{A}_1},\ub_t^{j,\mathcal{A}_1}) - f_{i,t}(\Mb_{1:h}^{j,t}|\Ab,\Bb,\{\wb\})\right]\\
        +&\sum_{t=1}^T\sum_{i=1}^n\left[f_{i,t}(\Mb_{1:h}^{j,t}|\Ab,\Bb,\{\wb\})-f_{i,t}(\Mb_{1:h}^{*}|\Ab,\Bb,\{\wb\})\right]\\
        +&\sum_{t=1}^T\sum_{i=1}^n\left[f_{i,t}(\Mb_{1:h}^{*}|\Ab,\Bb,\{\wb\}) - f_{i,t}(\xb_t^{\Kb^*,m}, \ub_t^{\Kb^*,m})\right]\\
        +&\sum_{t=1}^T\sum_{i=1}^n\left[f_{i,t}(\xb_t^{\Kb^*,m}, \ub_t^{\Kb^*,m}) - f_{i,t}(\xb_t^{\Kb^*}, \ub_t^{\Kb^*})\right],
    \end{split}    
    \end{equation}
    where each term can be bounded as follows:
    \begin{enumerate}
        \item Based on Lemma \ref{L: Deviation between actual costs and reparameterize costs}, we have
        \begin{equation}\label{Eq2: T Main Theorem}
        \begin{split}
            |f_{i,t}(\xb_t^{j,\mathcal{A}_1}, \ub_t^{j,\mathcal{A}_1}) - f_{i,t}(\Mb_{1:h}^{j,t}|\Ab,\Bb,\{\wb\})| \leq \frac{144 G^2\kappa_B^2\kappa^{10}W^4\sqrt{n}}{\gamma^3\sqrt{T}(1-\beta)}\log^{1/2}\left(\frac{2}{\gamma}\right).
        \end{split}    
        \end{equation}
        
        \item As $\Kb^*\in\Sc$, we have $\Ab + \Bb\Kb^* = \Hb\Lb\Hb^{-1}$, where $\Lb$ is diagonal and can be written as $\Lb = \sum_{l=1}^d \alpha_l \eb_l\eb_l^{\top}$, $\alpha_l\leq \frac{1}{3}, l\in[d]$. Then by setting $\Mb^*_i=\sigma_i^{1/4}\Kb^*\Hb\left(\sum_{i=1}^d \phi_i^{\top}\mu_{\alpha_j}\eb_j\eb_j^{\top}\right)$. Then by Lemma \ref{L: Bound on the inner product between a geometric vector and the eigenvector of Hankel matrix}, we have $\norm{\Mb^*_i}\leq \kappa^3\sqrt{\frac{2}{\gamma}}$, which implies $\norm{\Mb^*_{1:h}}\leq \kappa^3\sqrt{\frac{2h}{\gamma}}$. In addition, based on the proof of Lemma 5.2 in Hazan, it can be shown that $\Mb^*_{1:h}\in \mathcal{K}$. Now from Theorem \ref{T: Distributed regret bound}, we have 
        \begin{equation}\label{Eq3: T Main Theorem}
        \begin{split}
            &\sum_{t=1}^T\sum_{i=1}^n \left[f_{i,t}(\Mb_{1:h}^{j,t}|\Ab,\Bb,\{\wb\}) - f_{i,t}(\Mb^*_{1:h}|\Ab,\Bb,\{\wb\})\right]\\
            \leq &\frac{n}{2\eta}\norm{\Mb_{1:h}^1 - \Mb^*_{1:h}}^2 + T\left(\frac{8\eta n\sqrt{n}}{(1-\beta)} + 2\eta\right)\left(\frac{6G\kappa_B\kappa^5W^2\sqrt{m}h}{\gamma^2}\log^{1/4}\left(\frac{2}{\gamma}\right)\right)^2
        \end{split}
        \end{equation}


        \item Based on our choice of $h$ and Lemma \ref{L: Deviation between costs of SC and truncated linear policies}, we have $\forall i$
        \begin{equation}\label{Eq4: T Main Theorem}
        \begin{split}
            \sum_{t=1}^T\left[f_{i,t}(\Mb^*_{1:h}|\Ab,\Bb,\{\wb\}) - f_{i,t}(\xb_t^{\Kb^*,m}, \ub_t^{\Kb^*,m})\right]\leq \frac{\epsilon T}{2}.
        \end{split}
        \end{equation}

        \item With the selection of $m$ and Lemma \ref{L: Deviation between costs of normal and truncated linear policies}, we have $\forall i$
        \begin{equation}\label{Eq5: T Main Theorem}
        \begin{split}
            \sum_{t=1}^T\left[f_{i,t}(\xb_t^{\Kb^*,m}, \ub_t^{\Kb^*,m}) - f_{i,t}(\xb_t^{\Kb^*}, \ub_t^{\Kb^*})\right]\leq \frac{\epsilon T}{2}           
        \end{split}
        \end{equation}
    \end{enumerate}
    Substituting Equations \eqref{Eq2: T Main Theorem}, \eqref{Eq3: T Main Theorem}, \eqref{Eq4: T Main Theorem} and \eqref{Eq5: T Main Theorem} into Equation \eqref{Eq1: T Main Theorem}, we get
    \begin{equation}\label{Eq6: T Main Theorem}
    \begin{split}
        &J_T^j(\mathcal{A}_1) - \min_{\Kb}J_T(\Kb)\\
        \leq &\frac{n}{2\eta}\norm{\Mb_{1:h}^1 - \Mb^*_{1:h}}^2 + T\left(\frac{8\eta n\sqrt{n}}{(1-\beta)} + 2\eta\right)\left(\frac{6G\kappa_B\kappa^5W^2\sqrt{m}h}{\gamma^2}\log^{1/4}\left(\frac{2}{\gamma}\right)\right)^2\\
        + &nT\frac{144 G^2\kappa_B^2\kappa^{10}W^4\sqrt{n}}{\gamma^3\sqrt{T}(1-\beta)}\log^{1/2}\left(\frac{2}{\gamma}\right) + n\epsilon T,
    \end{split}
    \end{equation}
    where the final result is derived by replacing $m,h,\eta$ with our selections.

\end{proof}


\end{document}